\documentclass[12pt,reqno]{amsart}

\usepackage{amsmath}
\usepackage{amsfonts}
\usepackage{amscd}
\usepackage{amssymb}
\usepackage{amsthm}
\usepackage{mathdots}
\usepackage{mathtools}
\usepackage{color}
\usepackage{bm}
\usepackage{enumerate}
\usepackage{graphicx}
\usepackage[linktocpage=true]{hyperref}
\usepackage{mathrsfs}
\usepackage{caption}

\usepackage{microtype}

\usepackage{tikz,tikz-cd, color}

\usepackage{fullpage}
\usepackage{comment}

\newtheorem{thm}{Theorem}
\newtheorem*{thm*}{Theorem}
\newtheorem{lem}[thm]{Lemma}
\newtheorem{prop}[thm]{Proposition}
\newtheorem{cor}[thm]{Corollary}
\newtheorem{thm&defn}[thm]{Theorem \& Definition}
\newtheorem{defn}[thm]{Definition}

\theoremstyle{remark}

\numberwithin{equation}{section}
\numberwithin{thm}{section}

\newcommand{\cF}{\mathcal{F}}

\newcommand{\bF}{\mathbb{F}}
\newcommand{\bP}{\mathbb{P}}

\newcommand{\lExt}{{\mathcal{E}xt}}

\newcommand{\db}{\operatorname{D}^{\rm b}}

\DeclareMathOperator{\Ext}{Ext}

\DeclareMathOperator{\rk}{rk}
\DeclareMathOperator{\im}{im}
\DeclareMathOperator{\Aut}{Aut}
\DeclareMathOperator{\Pic}{Pic}

\title{Autoequivalences of Derived Categories of Moduli Spaces of Vector Bundles}

\author[]{Haotian Zuo}
\address{Department of Mathematics, University of Houston, Houston, TX, USA}
\email{hzuo@cougarnet.uh.edu}
\begin{document}

\begin{abstract}
    Let \(C\) be a smooth projective curve over an algebraically closed field of
characteristic zero. For the moduli space \(N(r,L)\) of stable vector bundles
on \(C\) of rank \(r\) with fixed determinant \(L\), we study the group of exact
autoequivalences of its bounded derived category. Combining the Bondal--Orlov
reconstruction theorem with the descriptions of \(\operatorname{Aut}(N(r,L))\)
due to Kouvidakis--Pantev and Newstead, we obtain an explicit
description of \(\operatorname{Aut}\db(N(r,L))\). We also construct extension correspondences between moduli spaces of vector
bundles of different ranks and use them to define natural Fourier--Mukai type
exact functors between their bounded derived categories.
\end{abstract}
\maketitle
\section{Introduction}
The moduli space of vector bundles over a nonsingular projective curve has been intensively studied for decades. In recent years, considerable attention has been devoted to the derived categories of these moduli spaces.

A central problem in the study of the structure of derived categories of algebraic varieties is to determine their groups of autoequivalences. However, for general algebraic varieties, the structure of the group of autoequivalences of the derived category is highly complicated, and they are known only in a few cases (such as for Fano varieties \cite{BO01} or K3 surfaces \cite{Orl97}).

Let $C$ be a nonsingular projective curve of genus $g\geq 2$ over an algebraically closed field $k$ of characteristic $0$. Suppose $r$ is an integer $\geq 2$ and $L$ is a line bundle of degree $d$ on $C$.
Let $M=M(r,d)=M_C(r,d)$ denote the moduli space of stable vector bundles of rank $r$ and degree $d$. 
Similarly, let $N=N(r,L)=N_C(r,L)\subset M$ denote the moduli of those with determinant $L$. 
We make the usual assumption 
$\gcd(r,d)=1$
such that $M$ and $N$ are projective, nonsingular, and fine moduli spaces.
We will study the bounded derived categories
\(\mathrm{D}^{\mathrm b}(M)\) and \(\mathrm{D}^{\mathrm b}(N)\) of coherent sheaves on \(M\) and \(N\), respectively. For a smooth projective variety \(X\), we write
\(\operatorname{Aut}\mathrm{D}^{\mathrm b}(X)\) for the group of exact autoequivalences of
\(\mathrm{D}^{\mathrm b}(X)\), with group law given by composition.

When the genus of \(C\) is at least \(3\), combining the Bondal--Orlov reconstruction theorem with the theorem of Kouvidakis--Pantev gives the following explicit description of \(\operatorname{Aut}D^b(N(r,L))\). 
\begin{thm}
    Suppose \(g\geq 3\). Then 
\[
\operatorname{Aut}\mathrm{D}^{\mathrm{b}}(N(r,L))
\cong
\begin{cases}
\mathbb Z^2\times \mathfrak G^\circ_{[r]}, & \text{if } r\neq 2,\\[6pt]
\mathbb Z^2\times \mathfrak G_{[2]}, & \text{if } r=2.
\end{cases}
\]
\end{thm}For the definition of the groups $\mathfrak{G}_{[r]}^{\circ}$ and $\mathfrak{G}_{[r]}$, see Definition \ref{autgp}. We also discuss some low-genera cases.

In rank two, Belmans--Galkin--Mukhopadhyay conjectured, independently of Narasimhan, an explicit description of the group
\(\operatorname{Aut}\mathrm{D}^{\mathrm b}(N(2,L))\) \cite{BGM23,Nar17}. This conjecture was proved by Tevelev--Torres and Tevelev \cite{TT24,Tev24}, using Thaddeus’ wall-crossing for stable pairs in an essential way \cite{Thad94}. Although Thaddeus' wall-crossing method is particularly effective in rank two,
its direct extension to higher rank appears to be rather complicated. In the
last section, we record a simple extension correspondence between moduli spaces
of different ranks. This correspondence produces a natural Fourier--Mukai type
exact functor.

Since the isomorphism class of $M$ or $N$ will not be changed after tensoring all stable vector bundles by a fixed line bundle, we assume $0<d<r$ to simplify our discussions.

In the $(\rk, \deg)$-plane, let $(s,e)$ be the integral point which is closest to the line through the origin and $(r,d)$ and lies in the following region:
\begin{align*}
    \Xi=\{(x,y)\mid 0< x<r,\ 0\leq y \leq d,\ y/x<d/r\}.
\end{align*}

Equivalently, since the distance from \((x,y)\) to the line through
the origin and \((r,d)\) is proportional to \(dx-ry\) in the region \(\Xi\),
the point \((s,e)\) is characterized by
\begin{equation}\label{closest}
    ds-re=1.
\end{equation}
Denote $t=r-s$ and $f=d-e$. Then the condition $(\ref{closest})$ is equivalent to $rf-td=1.$ This implies 
\(\gcd(s,e)=\gcd(t,f)=1. \)
We have
\begin{thm}\label{thm:main-deg}There is an exact functor 
    \[\Phi\colon \db\left(M(s,e )\times M(t,f)\right)\to \db(M(r,d)).\]
\end{thm}
For the definition of the functor $\Phi$, see (\ref{functor1}). 

We also have a parallel statement for the moduli spaces with determinants.
\begin{thm}\label{thm:main}Suppose $L_1$ and $L_2$ are line bundles of degrees $e$ and $f$ respectively such that $L_1\otimes L_2=L$. 
    There is an exact functor 
    \[\Psi\colon \db\left(N(s,L_1 )\times N(t,L_2)\right)\to \db(N(r,L)).\]
\end{thm}
For the definition of the functor $\Psi$, see (\ref{eq:key-functor-det}).

\section{Autoequivalences}
First we focus on the moduli spaces with fixed determinants. 
Since $N$ is Fano, the Bondal--Orlov reconstruction theorem \cite{BO01} implies that every exact autoequivalence of
\(\mathrm{D}^{\mathrm b}(N)\) is a composition of a shift, the push-forward by an automorphism of \(N\), and tensoring by a line bundle. Hence
\[\operatorname{Aut}\db\left(N\right)\cong \mathbb{Z}\times \left(\operatorname{Aut}\left(N\right)\ltimes\operatorname{Pic}\left(N\right)\right). \]
The subgroup $\mathbb{Z}$ is generated by the shift functor $[1]$, while elements of $\Aut(N)$ act as push-forwards.
On the other hand, it is well-known \cite[Theorem 1]{Ram73} that $\Pic(N)\cong \mathbb{Z}$.
We next describe $\Aut(N)$.
\begin{defn}\label{autgp} There are two subgroups of $\operatorname{Pic}(C)\times\operatorname{Aut}(C)$:
    \begin{align*}
  \mathfrak{G}^\circ_{[r]}&=\{\left(\Lambda,\sigma\right)\in \operatorname{Pic}(C)\times\operatorname{Aut}(C)\mid\Lambda^r\cong L\otimes \sigma^*L^{-1}\} ,\\
  \mathfrak{G}_{[r]}&=\{\left(\Lambda,\sigma\right)\in \operatorname{Pic}(C)\times\operatorname{Aut}(C)\mid\Lambda^r\cong L\otimes \sigma^*L^{\pm 1}\} .
\end{align*}

\end{defn}

The following theorem is due to Kouvidakis and Pantev \cite[Thm. B]{KP95}. 
\begin{thm}[Kouvidakis--Pantev]Suppose the genus of $C$ is $g\geq 3$ . Then the following map is an isomorphism:

\begin{enumerate}[(i)]
    \item If $r\neq 2$, the morphism 
    \[
        \mathfrak{G}_{[r]}^{\circ}\to\operatorname{Aut}\big(N(r,L)\big),\qquad
        (\Lambda,\sigma)\longmapsto \big(E\mapsto\sigma^*E\otimes\Lambda\big)
    \]
    is an isomorphism.
    
    \item If $r=2$, the morphism 
    \[
        \mathfrak{G}_{[r]}\to\operatorname{Aut}\big(N(r,L)\big),\qquad
        (\Lambda,\sigma)\longmapsto
        \left(
        E\longmapsto
        \begin{cases}
        \sigma^*E\otimes\Lambda, & \text{if }(\Lambda,\sigma)\in\mathfrak{G}_{[r]}^{\circ},\\[3pt]
        \sigma^*E^{\vee}\otimes\Lambda, & \text{if }(\Lambda,\sigma)\in\mathfrak{G}_{[r]}\setminus\mathfrak{G}_{[r]}^{\circ},
        \end{cases}
        \right)
    \]
    is an isomorphism.
\end{enumerate}
\end{thm}

Since \(\operatorname{Pic}(N(r,L))\cong \mathbb Z\) is generated by the ample theta
line bundle \(\Theta\), the natural action of \(\operatorname{Aut}(N(r,L))\) on
\(\operatorname{Pic}(N(r,L))\) is trivial. Indeed, for any \(f\in \operatorname{Aut}(N(r,L))\),
the line bundle \(f^*\Theta\) is again ample. Therefore \(f^*[\Theta]=[\Theta]\).

   Thus, when the genus of the curve $C$ is at least 3, the group of autoequivalences of the derived category $\db(N(r,L))$ can be explicitly described as follows:
\begin{cor}\label{cor}
The group $\operatorname{Aut}\mathrm{D}^{\mathrm{b}}(N(r,L))$  
is completely determined by $\mathbb{Z}$ and $\mathfrak{G}_{[r]}^{\circ}$ (or $\mathfrak{G}_{[r]}$):
\[
\operatorname{Aut}\mathrm{D}^{\mathrm{b}}(N(r,L))
\cong
\begin{cases}
\mathbb{Z}\times\big(\mathfrak{G}_{[r]}^{\circ}\times\mathbb{Z}\big)\cong\mathbb Z^2\times \mathfrak G^\circ_{[r]}, & \text{if } r\neq 2,\\[6pt]
\mathbb{Z}\times\big(\mathfrak{G}_{[2]}\times\mathbb{Z}\big)\cong\mathbb Z^2\times \mathfrak G_{[2]}, & \text{if } r=2.
\end{cases}
\]
\end{cor}
In the remainder of this section, we discuss the low-genus cases that are not covered by the above theorem.  

When the curve $C$ has genus $g=2$, Newstead showed in \cite{New68} by using explicit geometric methods that if $r=2$ and $d=\deg L$ is odd, then the following holds:

\begin{thm}[{\cite[Theorem~3]{New68}}]
The morphism
\[
\mathfrak{G}_{[2]}^{\circ} \longrightarrow \operatorname{Aut}\big(N(2,L)\big), \qquad
(\Lambda,\sigma)\longmapsto \big(E\longmapsto\sigma^*E\otimes\Lambda\big)
\]
is an isomorphism.
\end{thm}
Thus, unlike the case \(g\geq 3\) in rank two, Newstead's genus-two result does
not include an additional automorphism induced by dualization.
Therefore, we immediately obtain the following corollary:
\begin{cor}
Let $C$ be a smooth projective curve of genus $g=2$, and let $N(2,L)$ be the moduli space of stable vector bundles of rank $2$ and determinant $L$, 
where $\deg L$ is odd. Then the group of autoequivalences of the derived category $\mathrm{D}^{\mathrm{b}}(N(2,L))$ is given by
\[
\operatorname{Aut}\mathrm{D}^{\mathrm{b}}(N(2,L)) \cong \mathbb{Z} \times \big(\mathfrak{G}_{[2]}^{\circ} \times \mathbb{Z}\big)\cong\mathbb Z^2\times \mathfrak G^\circ_{[2]}.
\]
\end{cor}

However, when $g=2$ and $r>2$, the structure of the automorphism group $\operatorname{Aut}(N(r,L))$ is still unknown.

In contrast, the cases of curves of genus $g=0$ and $g=1$ have been well-understood by methods different from those for $g\ge 2$.
For completeness, we list the corresponding results below.

When $C$ is an elliptic curve of genus $g=1$, fixing a point $O\in C$ makes $(C,O)$ an elliptic curve. Tu proved in \cite[Theorem~3]{Tu93} that for any elliptic curve $C$,
the moduli space $N(r,L)$ is isomorphic to $\mathbb{P}^{h-1}$, where $h=\operatorname{gcd}(r,d)$. 
When \(\gcd(r,d)=1\), this gives \(N(r,L)\cong \mathbb{P}^0\), a single point.

When $C$ is a curve of genus $g=0$, i.e.\ $C\simeq\mathbb{P}^1$, Grothendieck showed that every vector bundle on $\mathbb{P}^1$ decomposes uniquely as a direct sum of line bundles on $\mathbb{P}^1$ \cite[Th\'eor\`eme 2.1]{Gro57}.
 In particular, every stable vector bundle on $\mathbb{P}^1$ is a line bundle, and there are no stable vector bundles of rank greater than 1. 
Therefore, the problem concerned here is also trivial.

\section{Proofs of Theorem 1.2 and Theorem 1.3}
Let us first prove the following key observation, which will also explain the choice of $(s,e)$. 

Let $r,s,t,d,e$, and $f$ be as in the introduction. 
Let \(M_1=M(s,e)\quad\mbox{and}\quad M_2=M(t,f).\)

\begin{lem}\label{lem:key} Suppose $E\in M_1$ and $F\in M_2$, and the following is a nontrivial extension:
\begin{align}\label{eq:ext}
    0\rightarrow E\rightarrow G\rightarrow F\rightarrow 0.
\end{align}
Then $G$ is stable. 
\end{lem}
\begin{proof}
Let $G_0\subsetneqq G$ be any proper sub-bundle, we need to show $\mu(G_0)< \mu(G)=d/{r}$. We can assume $G_0$ is stable with maximal slope.
 Consider the composition $G_0\hookrightarrow G\twoheadrightarrow F$ which we denote by $\varphi $. 
 We divide the argument into three cases.
 \begin{enumerate}[(i)]
 \item If $\varphi=0$, then $G_0\subset E$. If $G_0\neq E$, then by the stability of $E$, $\mu (G_0)< \mu (E),$ contradicts the choice of $G_0.$ Hence in this case $G_0=E$, and by the choice of $E$ we have $\mu(E)<\mu(G)$, so we have $\mu(G_0)<\mu(G)$, which shows that $G$ is stable.
     \item If $\varphi$ is neither zero nor surjective, we denote $I=\im\phi$. On the one hand,  $\mu(G_0)< \mu(I)$ due to our choice of $G_0$. On the other hand, \(I\subset F\) is a proper subbundle. Let
\(b=\operatorname{rk}I\) and \(m=\deg I\). Since \(F\) is stable,
\(
\frac mb<\frac ft.
\)
By the choice of \((s,e)\), we have \(rf-td=1\), hence
\(
\frac ft=\frac dr+\frac{1}{rt}.
\)
If \(\frac mb\geq \frac dr\), then \(rm-bd\geq 1\), and therefore
\[
\frac mb-\frac dr=\frac{rm-bd}{rb}\geq \frac{1}{rb}>
\frac{1}{rt},
\]
because \(b<t\). This contradicts \(\frac mb<\frac ft\). Hence
\(
\mu(I)<\frac dr=\mu(G).
\)
Therefore \(\mu(G_0)<\mu(G)\).
     \item If $\varphi$ is surjective, it is not injective, otherwise the extension splits. We can form the following commutative diagram:
     \begin{equation*}
         \begin{tikzcd}
             0 \arrow[r] & E_0 \arrow[r] \arrow[d] & G_0 \arrow[r] \arrow[d] & F \arrow[r] \arrow[d,equal] &0\\
            0 \arrow[r] & E \arrow[r]  & G \arrow[r]  & F \arrow[r] &0. 
         \end{tikzcd}
     \end{equation*}
     We also have $\mu(G_0)<\mu(G)$ due to the choice of $(s,e)$.
 \end{enumerate}

 \end{proof}

We need the following property of the relative Ext-sheaf.
\begin{prop}[{\cite[1.7]{Gro62}}]
    Consider the relative Ext sheaf $\mathcal{E}xt^i_{\pi}(\mathcal{F}, \mathcal{G})$ defined as above. 
    It is a coherent sheaf on $S$, and for any point $s\in S$, its stalk is given by
    \[
        \mathcal{E}xt^i_{\pi}(\mathcal{F}, \mathcal{G})_s = \operatorname{Ext}^i_{X_s}(\mathcal{F}_s, \mathcal{G}_s),
    \]
    where $X_s = \pi^{-1}(s)$ is the fiber over $s\in S$, and $\mathcal{F}_s$, $\mathcal{G}_s$ are the restrictions of 
    $\mathcal{F}$ and $\mathcal{G}$ to $X_s$, respectively.
\end{prop}

For $i=1,2$, let
\[
    \mathcal{E}_i \to M_i \times C
\]
be the universal bundles on the fine moduli spaces $M_i$.  
Let $\pi_{ij}$ $(i,j=1,2,3)$ denote the projections from $M_1\times M_2\times C$.  
Then we consider the sheaf on $M_1\times M_2$ defined by
\[
    \mathbb{E} = \mathcal{E}xt^1_{\pi_{12}}\!\left(\pi_{23}^*\mathcal{E}_2, \pi_{13}^*\mathcal{E}_1\right).
\]

\begin{lem}
   The sheaf $\mathbb{E}$ is locally free on $M_1\times M_2$.
\end{lem}

\begin{proof}
    It suffices to note that for any closed point $x=(E,F)\in M_1\times M_2$, 
    the stalk at $x$ is
    \[
        \mathcal{E}xt^1_{\pi_{12}}\!\left(\pi_{23}^*\mathcal{E}_2, \pi_{13}^*\mathcal{E}_1\right)_x
        = \operatorname{Ext}^1_C(F,E).
    \]
    This is a $k$-vector space.  
    By the Riemann–Roch theorem, we have
    \begin{align*}
        \dim \operatorname{Hom}_C(F,E) - \dim \operatorname{Ext}^1_C(F,E)=e(r-s) -(d-e)s + s(r-s)(1-g).
    \end{align*}
    Since $E$ and $F$ are chosen such that $\mu(F) > \mu(E)$, we have $\operatorname{Hom}_C(F,E)=0$.  
    Thus, $\operatorname{Ext}^1_C(F,E)$ has constant dimension
    \[
        \dim \operatorname{Ext}^1_C(F,E) = (d-e)s-e(r-s) + s(r-s)(g-1),
    \]
    for all $x\in (E,F)\in M_1\times M_2$.  
    By Grauert’s theorem \cite[Corollary~12.9]{Har77}, 
    $\mathbb{E}$ is locally free on $M_1\times M_2$.
\end{proof}

Now, let $p\colon Z=\mathbb{P}(\mathbb{E})\to M_1\times M_2$ be the projection from the projectivization.  
At the point $(E,F)$, its fiber $\mathbb{P}\!\operatorname{Ext}^1_C(F,E)$ 
corresponds to the projective space of extension classes
\[
    0 \longrightarrow E \longrightarrow G \longrightarrow F \longrightarrow 0.
\]
With the convention that \(\mathbb P(\mathbb{E})\) parametrizes one-dimensional
subspaces in the fibers of \(\mathbb{E}\), the tautological line bundle
\(\mathcal O_Z(-1)\subset p^*\mathbb{E}\) determines a universal extension on
\(Z\times C\). By Lemma \ref{lem:key}, the middle term of every fiber of this extension is
stable of rank \(r\) and degree \(d\). Hence the universal property of the fine
moduli space \(M(r,d)\) gives a morphism \(q:Z\to M(r,d)\), which maps the universal extension (\ref{eq:ext}) to a point $G\in M$.

Thus we have a diagram
\begin{equation*}
    \begin{tikzcd}[column sep=small]
& Z \arrow[dl,"p"'] \arrow[dr,"q"] & \\
M_{1}\times M_2  & & M
\end{tikzcd}.
\end{equation*}

In particular, $Z$ can be viewed as a correspondence between $M_1\times M_2$ and $M$. 

Let
\(
i=(p,q):Z\rightarrow \left(M_{1}\times M_2\right)\times M.
\)
The correspondence \(Z\) defines a Fourier--Mukai kernel
\[
\mathcal K=i_*\mathcal O_Z\in D^b\left(\left(M_{1}\times M_2\right)\times M\right).
\]

Let \(\pi_1:\left(M_{1}\times M_2\right)\times M\to \left(M_{1}\times M_2\right)\) and \(\pi_2:\left(M_{1}\times M_2\right)\times M\to M\) be the two projections.
The Fourier--Mukai functor associated with
\(\mathcal K=i_*\mathcal O_Z\) is
\[\Phi(-):=
\Phi_{\mathcal K}(-)
=
R\pi_{2*}\bigl(\pi_1^*(-)\otimes^{\mathbf L}\mathcal K\bigr).
\]
By the projection formula applied to \(i:Z\to \left(M_{1}\times M_2\right)\times M\), we obtain a natural
isomorphism
\[
\Phi (-)\cong Rq_*Lp^*(-).
\]
Since \(p\) is a projective bundle, it is flat, and hence \(Lp^*=p^*\). Therefore
\begin{equation}\label{functor1}
    \Phi \cong Rq_*\circ p^*.
\end{equation}

Hence there exists an exact functor $\Phi\colon \db(M_1\times M_2)\to \db(M)$, which proves Theorem \ref{thm:main-deg} .
\medskip

We next consider the moduli space with fixed determinants.
Suppose $L_1$ and $L_2$ are line bundles of degrees $e$ and $f$ respectively such that $L_1\otimes L_2=L$. 
Let \[N_1=N(s,L_1)\quad\mbox{and}\quad N_2=N(t,L_2).\]

 For $i=1,2$, let 
 \begin{align*}
     \cF_i\to N_i\times C
 \end{align*}
be the universal bundle over the fine moduli space $N_i$. Denote the projections from $M_1\times M_2\times C$ as $\pi_{ij}$ for $i,j=1,2,3$, by abuse of notation. 
The coherent sheaf 
\begin{align*}
    \bF=\lExt^1_{\pi_{12}}(\pi_{23}^*\cF_2,\pi_{13}^*\cF_1)
\end{align*}
on $N_1\times N_2$ is also locally free.
By abuse of notation, 
    let $p\colon Z^{\det}=\bP(\bF)\to N_1\times N_2$
be the canonical projection of the projective bundle whose fiber over $(E,F)$ is $\bP \Ext^1(F,E)$. 
According to Lemma~\ref{lem:key} and the previous discussions, there is a modular map 
    $q\colon Z^{\det}\to N$.
Thus we have a diagram
\begin{equation*}
    \begin{tikzcd}[column sep=small]
& Z^{\det} \arrow[dl,"p"'] \arrow[dr,"q"] & \\
N_{1}\times N_2  & & N
\end{tikzcd}
\end{equation*}
Again, $Z^{\det}$ can be viewed as a correspondence between $N_1\times N_2$ and $N$. 
Thus we can define the Fourier--Mukai functor
\begin{align}\label{eq:key-functor-det}
    \Psi=Rq_{*}\circ p^*\colon \db(N_1\times N_2) &\to  \db(N) ,   
\end{align}
which is also an exact functor. This proves Theorem \ref{thm:main}.  Notice the definitions of $Z^{\det}$ and $\Psi$ depend on a choice of a pair of line bundles $(L_1,L_2)$.  

This construction gives a natural Fourier--Mukai type functor associated with
the extension correspondence.
\medskip

\noindent\textbf{Acknowledgements.}
This paper grew out of part of the author's master's thesis written at Tongji University. The author is grateful to Yinbang Lin for guidance and encouragement during the thesis work.
\medskip

\bibliography{main}

@article{Ram73,
  author  = {Ramanan, S.},
  title   = {The moduli spaces of vector bundles over an algebraic curve},
  journal = {Math. Ann.},
  volume  = {200},
  year    = {1973},
  pages   = {69--84},
  doi     = {10.1007/BF01578292}
}

@book{Har77,
  author    = {Hartshorne, Robin},
  title     = {Algebraic Geometry},
  series    = {Graduate Texts in Mathematics},
  volume    = {52},
  publisher = {Springer},
  address   = {New York},
  year      = {1977},
  doi       = {10.1007/978-1-4757-3849-0}
}

@article{KP95,
  author  = {Kouvidakis, Alexis and Pantev, Tony},
  title   = {The automorphism group of the moduli space of semistable vector bundles},
  journal = {Math. Ann.},
  volume  = {302},
  year    = {1995},
  pages   = {225--268}
}

@article{BO01,
  author  = {Bondal, Alexei and Orlov, Dmitri},
  title   = {Reconstruction of a variety from the derived category and groups of autoequivalences},
  journal = {Compos. Math.},
  volume  = {125},
  year    = {2001},
  number  = {3},
  pages   = {327--344}
}

@article{TT24,
  author  = {Tevelev, Jenia and Torres, Sofia},
  title   = {The {BGMN} conjecture via stable pairs},
  journal = {Duke Math. J.},
  volume  = {173},
  year    = {2024},
  number  = {18},
  pages   = {3495--3557},
  doi     = {10.1215/00127094-2024-0016}
}

@article{BGM23,
  author  = {Belmans, Pieter and Galkin, Sergey and Mukhopadhyay, Swarnava},
  title   = {Decompositions of moduli spaces of vector bundles and graph potentials},
  journal = {Forum Math. Sigma},
  volume  = {11},
  year    = {2023},
  pages   = {e24},
  doi     = {10.1017/fms.2023.14}
}

@article{Nar17,
  author  = {Narasimhan, M. S.},
  title   = {Derived categories of moduli spaces of vector bundles on curves},
  journal = {J. Geom. Phys.},
  volume  = {122},
  year    = {2017},
  pages   = {53--58},
  doi     = {10.1016/j.geomphys.2017.01.018}
}

@article{Thad94,
  author  = {Thaddeus, Michael},
  title   = {Stable pairs, linear systems and the {V}erlinde formula},
  journal = {Invent. Math.},
  volume  = {117},
  year    = {1994},
  number  = {2},
  pages   = {317--353},
  doi     = {10.1007/BF01232244}
}

@article{Tev24,
  author        = {Tevelev, Jenia},
  title         = {Braid and Phantom},
  journal       = {arXiv preprint},
  year          = {2024},
  eprint        = {2304.01825},
  archivePrefix = {arXiv}
}

@book{Gro62,
  author    = {Grothendieck, Alexander},
  title     = {Fondements de la g{\'e}om{\'e}trie alg{\'e}brique},
  subtitle  = {Extraits du S{\'e}minaire Bourbaki, 1957--1962},
  publisher = {Secr{\'e}tariat math{\'e}matique},
  address   = {Paris},
  year      = {1962}
}

@article{Orl97,
  author  = {Orlov, Dmitri O.},
  title   = {Equivalences of derived categories and {$K3$} surfaces},
  journal = {J. Math. Sci.},
  volume  = {84},
  year    = {1997},
  number  = {5},
  pages   = {1361--1381},
  doi     = {10.1007/BF02399195}
}

@article{New68,
  author  = {Newstead, Peter E.},
  title   = {Stable bundles of rank 2 and odd degree over a curve of genus 2},
  journal = {Topology},
  volume  = {7},
  year    = {1968},
  pages   = {205--215}
}

@article{Tu93,
  author  = {Tu, Loring W.},
  title   = {Semistable bundles over an elliptic curve},
  journal = {Adv. Math.},
  volume  = {98},
  year    = {1993},
  number  = {1},
  pages   = {1--26}
}

@article{Gro57,
  author  = {Grothendieck, Alexander},
  title   = {Sur la classification des fibr{\'e}s holomorphes sur la sph{\`e}re de {R}iemann},
  journal = {Amer. J. Math.},
  volume  = {79},
  year    = {1957},
  number  = {1},
  pages   = {121--138}
}
\bibliographystyle{alpha}

\end{document}